\RequirePackage[l2tabu,orthodox]{nag}

\documentclass[11pt]{amsart}

\usepackage[T1]{fontenc}
\usepackage{lmodern}

\usepackage{amsmath,amssymb,latexsym,soul,cite}
\usepackage{color,enumitem,graphicx, tikz, mathtools, microtype}

\usepackage[left=2.65cm,right=2.65cm,top=3cm,bottom=3cm]{geometry}
\usepackage{color,enumitem,graphicx}
\usepackage[colorlinks=true,urlcolor=green,citecolor=green,linkcolor=green,linktocpage,pdfpagelabels,bookmarksnumbered,bookmarksopen]{hyperref}
\usepackage[english]{babel}
\usepackage[initials]{amsrefs}

\numberwithin{equation}{section}
\newtheorem{theorem}{Theorem}[section]
\newtheorem{lemma}[theorem]{Lemma}

\newtheorem{proposition}[theorem]{Proposition}
\newtheorem{corollary}[theorem]{Corollary}

\theoremstyle{remark}
\newtheorem{remark}[theorem]{Remark}

\theoremstyle{definition}
\newtheorem{definition}[theorem]{Definition}

\DeclareMathOperator{\tr}{tr}

\title{On some nonlinear fractional equations involving the Bessel operator}
\author{Simone Secchi}
\address{Dipartimento di Matematica e Applicazioni, Universit\`a degli Studi di Milano Bicocca, via Cozzi 55, 20255 Milano, Italy}
\email{Simone.Secchi@unimib.it}
\thanks{Supported by FIRB 2012 ``Dispersive equations: Fourier analysis and variational methods'' and by PRIN 2012 ``Aspetti
variazionali e perturbativi nei problemi differenziali
nonlineari''}
\date{\today}
\dedicatory{Dedicated to Francesca}

\begin{document}

\begin{abstract}
Under different assumptions on the potential functions $b$ and $c$, we study the fractional equation $\left( I-\Delta \right)^{\alpha} u = \lambda b(x) |u|^{p-2}u+c(x)|u|^{q-2}u$ in $\mathbb{R}^N$. Our existence results are based on compact embedding properties for weighted spaces.
\end{abstract}
\maketitle

\section{Introduction}

In this paper we provide some existence results for a class of nonlinear fractional equations of the form
\begin{equation} \label{eq:1}
\left( I-\Delta \right)^{\alpha} u = \lambda b(x) |u|^{p-2}u+c(x)|u|^{q-2}u \quad \text{in $\mathbb{R}^N$}
\end{equation}
where $0<\alpha <1$, $p$ and $q$ belong to the interval $(0,2_\alpha^*)$ with $2_\alpha^*=2N/(N-2\alpha)$. 
We will assume throughout that
\begin{description}
 \item[(H)] The potential
functions $b$ and $c$ are continuous and bounded. 
\end{description}
Equations involving the nonlocal operator $(I-\Delta)^\alpha$ arise in the study of standing waves $\psi=\psi(t,x)$ for Schr\"{o}dinger--Klein--Gordon equations of the form
\begin{equation*}
  \mathrm{i} \frac{\partial \psi}{\partial t} = 	(I-\Delta)^\alpha \psi - f(x,\psi), \quad (t,x) \in \mathbb{R} \times \mathbb{R}^N
\end{equation*}
which describe the behavior of bosons. We refer to \cites{Laskin1, Laskin2} for a physical introduction to these fractional equations.

The case $\alpha=1$ corresponds to the classical Schr\"{o}dinger equation $-\Delta u + u = f(x,u)$, 
and we cannot review the huge literature here. 
In the last years the so-called fractional Laplacian $(-\Delta)^\alpha$ has become 
popular in the community of Nonlinear Analysts: we refer the interested reader to the 
recent survey \cite{Bucur} and to the references therein. 

The most physically relevant fractional case is $\alpha=1/2$, and the corresponding Bessel operator, 
or more precisely the operator $\sqrt{-\Delta+m^2}-m$, goes under the name of \emph{relativistic 
Schr\"{o}dinger operator}.

It is interesting to quote a sentence from \cite{Carmona}*{page 119}:
\begin{quote}
One of the fundamental mathematical problems in proving the stability
(or instability) of matter is to estimate the infimum of the spectrum of the
operator $H=H_0+V$ when the numbers $N$ and/or $M$ become large. In this
asymptotic regime, the free Hamiltonian $H_0 = F(p)=\sqrt{-\Delta+m^2}-m$ can be abandoned and replaced by $H_0=|p|$. Indeed, the difference between these two operators remains bounded and the asymptotic result needed to prove
the stability of matter can be proved using either one of these free
Hamiltonians. Obviously, the scaling properties of the function $F(p) = |p|$
attracted early investigations of the corresponding pseudo-differential
operator $H_0$ especially because its scaling is related to the scaling of the
Coulomb potential. Also, very fine estimates on its Green's function are
available.  These are the technical reasons why $F(p) = |p|$ is preferred to $\sqrt{p^2+m^2}-m$.
\end{quote}

We will come back to the issue of scaling properties in the last Section. R. Frank \emph{et al.} in \cite{Frank}*{page 5} promise to extend part of their spectral theory for the fractional Laplacian to the pseudodifferential operator $(-\Delta+m^2)^{\alpha}$.

\medskip

The equation $\sqrt{I-\Delta}\ u = f(u)$ was studied in \cite{Tan} by means of a Dirichlet-to-Neumann 
local realization that was extended to any $\alpha \in (0,1)$ by Fall \emph{et al.} in \cite{Fall}. The case \(\alpha=1/2\) was also studied in \cites{CingolaniSecchi1,CingolaniSecchi2,CZN1,CZN2} 
with a non-local convolutions term on the right-hand side.

On the contrary, very few papers deal with equation (\ref{eq:1}) for arbitrary~$\alpha \in (0,1)$. 
The very recent paper \cite{FelmerVergara} deals with the case in which the nonlinearity 
$f=f(x,u)$ is essentially of the form $\bar{f}(u)+a(x)(|u|+|u|^p)$ with $\lim_{|x| \to +\infty} a(x)=0$, 
in the spirit of \cite{Rabinowitz}.

In this paper we limit ourselves to a somehow particular class of nonlinearities as in (\ref{eq:1}), and we provide a few existence results under different assumptions on the two potential functions $b$ and $c$. 

Section \ref{sec:2} contains some preliminaries on the functional setting that we use to solve (\ref{eq:1}). 
 In Section \ref{sec:3} we begin with the case $p>2$, $q>2$, $b =0$ identically and $c$ positive and ``vanishing at infinity'' in an appropriate sense. In Section \ref{sec:4} we allow $b$ to tend to a positive constant at infinity, and in Section \ref{sec:5} we consider the concave-convex case $1<p<2<q$ with sign-changing potentials $b$ and $c$.

\bigskip

\noindent \textbf{Notation.}
\begin{enumerate}
 \item The letter $C$ will stand for a generic positive constant that
   may vary from line to line.
 \item The symbol $\|\cdot \|_p$ will be reserved for the norm in
   $L^p(\mathbb{R}^N)$.
 \item The operator $D$ will be reserved for the (Fr\'{e}chet)
   derivative, also for functions of a single real variable.
 \item The symbol $\mathcal{L}^N$ will be reserved for the Lebesgue
   $N$-dimensional measure.
 \item The symbol $\complement A$ will denote the complement of the
   subset $A$ (usually in $\mathbb{R}^N$).
 \item $\mathbb{R}_{+}^{N+1} = \left\{ (x,y) \in \mathbb{R}^N \times [0,+\infty) \right\}$
 \item For a real-valued function $f$, we set $f^{+}=\max\{f,0\}$, the positive part of $f$. The negative part of $f$ is defined similarly.
 \item The Fourier transform of a function $f$ will be denoted by $\mathcal{F}u$.
\end{enumerate}

\section{Preliminaries and functional setting} \label{sec:2}
For $\alpha>0$ we introduce the \emph{Bessel function space}
 \[
  L^{\alpha,2}(\mathbb{R}^N) = \left\{ f \colon f=G_\alpha * g
  \ \text{for some $g \in L^2(\mathbb{R}^N)$} \right\},
 \]
where the Bessel convolution kernel is defined by
\begin{equation} \label{eq:G}
G_\alpha (x) = 
\frac{1}{(4 \pi )^{\alpha /2}\Gamma(\alpha/2)} \int_0^\infty \exp \left( -\frac{\pi}{t} |x|^2 \right) \exp \left( -\frac{t}{4\pi} \right) t^{\frac{\alpha - N}{2}-1} \, dt
\end{equation}
The norm of this Bessel space is $\|f\| = \|g\|_2$ if $f=G_\alpha *
g$. The operator $(I-\Delta)^{-\alpha} u = G_{2\alpha} *u$ is usually
called Bessel operator of order $\alpha$.

In Fourier variables the same operator reads
\begin{equation} \label{eq:19}
 G_\alpha = \mathcal{F}^{-1} \circ \left( \left(1+|\xi|^2 \right)^{-\alpha /2} \circ \mathcal{F} \right),
\end{equation}
so that
\[
 \|f\| = \left\| (I-\Delta)^{\alpha /2} f \right\|_2.
\]
For more detailed information, see \cites{Adams, Stein} and the references therein.
\begin{remark}
In the paper \cite{Fall} the pointwise formula
\begin{equation} \label{eq:21}
(I-\Delta)^\alpha u(x) = 
c_{N,\alpha} \operatorname{P.V.} \int_{\mathbb{R}^N} \frac{u(x)-u(y)}{|x-y|^{\frac{N+2\alpha}{2}}} K_{\frac{N+2\alpha}{2}}(|x-y|) \, dy + u(x)
\end{equation}
was derived for functions $u \in C_c^2(\mathbb{R}^N)$. 
Here $c_{N,\alpha}$ is a positive constant depending only on $N$ and $\alpha$, 
P.V. denotes the principal value of the singular integral, and $K_\nu$ is the modified Bessel 
function of the second kind with order $\nu$ (see \cite{Fall}*{Remark 7.3} for more details). 
Since a closed formula for $K_\nu$ is unknown, equation (\ref{eq:21}) is not particularly useful 
for our purposes.
\end{remark}
We recall the embedding properties of Bessel spaces (see \cites{Felmer,Stein,Strichartz}).

\begin{theorem} 
\label{th:1}
\begin{enumerate}
\item $L^{\alpha,2}(\mathbb{R}^N) = W^{\alpha,2}(\mathbb{R}^N) =
  H^\alpha (\mathbb{R}^N)$.
\item If $\alpha \geq 0$ and $2 \leq q \leq 2_\alpha^*=2N/(N-2\alpha)$, then
  $L^{\alpha,2}(\mathbb{R}^N)$ is continuously embedded into $L^q(\mathbb{R}^N)$; if $2 \leq q < 2_\alpha^*$ then the embedding is locally compact.
\item Assume that $0 \leq \alpha \leq 2$ and $\alpha > N/2$. If
  $\alpha -N/2 >1$ and $0< \mu \leq \alpha - N/2-1$, then
  $L^{\alpha,2}(\mathbb{R}^N)$ is continuously embedded into
  $C^{1,\mu}(\mathbb{R}^N)$. If $\alpha -N/2 <1$ and $0 < \mu \leq
  \alpha -N/2$, then $L^{\alpha,2}(\mathbb{R}^N)$ is continuously
  embedded into $C^{0,\mu}(\mathbb{R}^N)$.
\end{enumerate}
\end{theorem}
\begin{remark}
Although the Bessel space $L^{\alpha,2}(\mathbb{R}^N)$ is topologically undistinguishable from the Sobolev fractional space $H^\alpha(\mathbb{R}^N)$, we will not confuse them, since our equation involves the Bessel norm.
\end{remark}
We collect here a couple of technical lemmas taken from the remarkable paper \cite{Palatucci}.

\begin{lemma} \label{lem:1}
	Let $0<\alpha < N/2$ and let $u \in
        L^{\alpha,2}(\mathbb{R}^N)$. Let $\varphi \in C_0^\infty
        (\mathbb{R}^N)$ and for each $R>0$ let
        $\varphi_R(x)=\varphi(R^{-1} x)$. Then $\lim_{R \rightarrow 0}
        \varphi_R u =0$ in $L^{\alpha,2}(\mathbb{R}^N)$. If, in
        addition, $\varphi$ equals one in a neighborhood of the
        origin, then $\lim_{R \rightarrow +\infty} \varphi_R u =0$ in
        $L^{\alpha,2}(\mathbb{R}^N)$.
\end{lemma}
\begin{proof}
  Since $L^{\alpha,2}(\mathbb{R}^N)$ is equivalent to
  $H^s(\mathbb{R}^N)$, the proof of \cite{Palatucci}*{Lemma 5} carries
  over with only minor modifications.
\end{proof}

\begin{lemma} \label{lem:2}
	Let $0<\alpha<N/2$, let $\Omega$ be a bounded open subset of
        $\mathbb{R}^N$ and let $\varphi \in
        C_0^\infty(\mathbb{R}^N)$. Then the commutator
        $[\varphi,(I-\Delta)^{\alpha/2}] \colon
        L^{\alpha,2}(\mathbb{R}^N) \to L^2(\mathbb{R}^N)$ is a compact
        operator.
\end{lemma}
\begin{proof}
It suffices to remark that the proof of \cite{Palatucci}*{Lemma 6} actually contains a proof of our statement.
\end{proof}

\begin{definition}
 We say that $u \in L^{\alpha,2}(\mathbb{R}^N)$ is a weak solution to (\ref{eq:1}) if
\[
  \int_{\mathbb{R}^N} (I-\Delta)^{\alpha /2} u \ (I-\Delta)^{\alpha/2} v \, dx = \int_{\mathbb{R}^N} b(x)|u|^{p-2}uv \, dx + \int_{\mathbb{R}^N} c(x) |u|^{q-2}uv \, dx
\]
for all $v \in L^{\alpha,2}(\mathbb{R}^N)$, or, equivalently,
\[
  \int_{\mathbb{R}^N} (1+|\xi|^2)^\alpha \mathcal{F}u(\xi) \mathcal{F}v(\xi) \, d\xi = \int_{\mathbb{R}^N} b(x)|u|^{p-2}uv \, dx + \int_{\mathbb{R}^N} c(x) |u|^{q-2}uv \, dx
\]
\end{definition}
It is readily seen that weak solutions to (\ref{eq:1}) correspond to critical points of the differentiable functional $J \colon L^{\alpha,2}(\mathbb{R}^N) \to \mathbb{R}$ defined by
\[
J(u)=\frac{1}{2} \|(I-\Delta)^{\alpha /2}u\|_2^2 - \frac{\lambda}{p} \int_{\mathbb{R}^N} b(x) |u|^p\, dx - \frac{1}{q} \int_{\mathbb{R}^N} c(x)|u|^q \, dx
\]

\section{Potentials vanishing at infinity} \label{sec:3}

In this section with treat the particular case of equation (\ref{eq:1}) with $b \equiv 0$, or
$\lambda=0$. We can actually deal with a more
general model,
\begin{equation} \label{eq:2}
 (I-\Delta)^\alpha u = c(x) f(u) \quad\text{in $\mathbb{R}^N$}
\end{equation}

\begin{definition}
We say that a function $c \in L^\infty(\mathbb{R}^N)$ satisfies the compactness condition (K) if 
\[
 \lim_{r \to +\infty} \int_{A_n \cap \complement B(0,r)} c(x)\, dx =0
 \] 
 uniformly with respect to $n$ whenever $\{A_n\}_n$ is a sequence of Borel sets such that 
 \[
 \sup_n \mathcal{L}^N (A_n)<\infty.
 \]
\end{definition}
Following \cite{Alves}, we assume
\begin{description}
 \item[(c1)] $c>0$ everywhere.
 \item[(c2)] $c$ satisfies the condition (K).
\end{description}
Our assumptions on the nonlinearity $f$ read as follows: $f$ is continuous and
\begin{description}
 \item[(f1)] $\limsup_{s \to 0} f(s)/s =0$.
 \item[(f2)] $\limsup_{s \to +\infty} f(s)/s^{2_\alpha^*-1}=0$.
 \item[(f3)] The map $s \mapsto s^{-1}f(s)$ is increasing and the primitive $F$ of $f$ satisfies $\lim_{s \to +\infty}
 F(s)/s^2=+\infty$.
\end{description}
\begin{remark} \label{rem:3}
Since $G_\alpha>0$, it follows easily that $(I-\Delta)^\alpha$
satisfies a maximum principle. Therefore, we will look for positive solutions and assume that $f(s)=0$
for every $s \leq 0$.
\end{remark}

The Euler functional associated to (\ref{eq:2}) is $J \colon
L^{\alpha,2}(\mathbb{R}^N) \to \mathbb{R}$, defined by
\begin{equation*}
 J(u) = \frac{1}{2} \|(I-\Delta)^{\alpha/2}\|_2^2 -
 \int_{\mathbb{R}^N} c(x) F(u)\, dx.
\end{equation*}
The following lemma can be proved by standard techniques.
\begin{lemma}
 The functional $J$ is of class $C^1$ has the mountain-pass geometry.
\end{lemma}
As a consequence, there exists (see \cite{Cerami}) a Cerami sequence $\{u_n\}_n \subset
L^{\alpha,2}(\mathbb{R}^N)$ for $J$, i.e. a sequence such that
\begin{equation} \label{eq:6}
 \lim_{n \to +\infty} J(u_n)=c \quad\text{and}\quad \lim_{n \to
   +\infty} (1+\|u_n\|)DJ(u_n)=0,
\end{equation}
where
\[
 c_{\mathrm{mp}} = \inf_{\gamma \in \Gamma} \sup_{0 \leq t \leq 1} J(\gamma(t))
\]
and
\[
 \Gamma = \left\{ \gamma \in C([0,1],L^{\alpha,2}\mathbb{R}^N) \mid \gamma (0)=0, \ J(\gamma(1))<0 \right\}.
\]
In view of Remark \ref{rem:3}, we may assume that $u_n \geq 0$ for every $n \in \mathbb{N}$.
%
\begin{definition}
For a positive and measurable function $w \colon \mathbb{R}^N \to [0,+\infty)$  we denote
\[
   L^q(w\, d\mathcal{L}^N) = \left\{ u\colon \mathbb{R}^N \to \mathbb{R} \mid \int |u(x)|^q w(x)\, dx <\infty \right\}               
\]
\end{definition}
\begin{proposition} \label{prop:1}
 The space $L^{\alpha,2}(\mathbb{R}^N)$ is compactly embedded into $L^q(c\, d\mathcal{L}^N)$ for every $2<q<2_\alpha^*$.
\end{proposition}
\begin{proof}
 We fix $q \in (2,2_\alpha^*)$ and pick $\varepsilon>0$. There exist $s_0<s_1$ such that
 \[
  c(x)|s|^q \leq \varepsilon C \left(s^2+|s|^{2_\alpha^*} \right) + Cc(x) \chi_{[s_0,s_1]}(|s|) |s|^{2_\alpha^*}
 \]
for every $s \in \mathbb{R}$. Integrating,
\[
 \int_{\complement B(0,r)} c(x)|u(x)|^q \, dx \leq \varepsilon C Q(u) + C \int_{A \cap \complement B(0,r)}
 c(x)\, dx
\]
for every $u \in L^{\alpha,2}(\mathbb{R}^N)$, where $Q(u)=\|u\|_2^2 +
\|u\|_{2_\alpha^*}^{2_\alpha^*}$ and $A = \left\{ x \in \mathbb{R}^N
\mid s_0 \leq |u(x)| \leq s_1 \right\}$.

If $v_n \to 0$ weakly in $L^{\alpha,2}(\mathbb{R}^N)$, then 
\[
\left \|(I-\Delta)^{\alpha/2}v_n \right\|_2^2 \leq M_1 \quad \text{and} \quad \|v_n\|_{2_\alpha^*}^{2_\alpha^*} \leq M_1
\]
for some $M_1>0$. In particular, the sequence $\{Q(v_n)\}_n$ is bounded in $\mathbb{R}$. On the other hand,
if we set 
\[
 A_n = \left\{ x \in \mathbb{R}^N \mid s_0 \leq |v_n(x)| \leq s_1 \right\},
\]
then 
\[
 s_0^{2_\alpha^*} \mathcal{L}^N(A_n) \leq \int_{A_n} |v_n|^{2_\alpha^*} \, dx \leq M_1
\]
for all $n \in \mathbb{N}$. This shows that $\sup_n \mathcal{L}^N(A_n)<\infty$ and from our assumption on $c$ there exists
a number $r >0$ such that
\[
 \int_{A_n \cap \complement B(0,r)} c(x)\, dx < \frac{\varepsilon}{s_1^{2_\alpha^*}}.
\]
We are ready to conclude that
\[
 \int_{B(0,r)} c(x)|v_n(x)|^q \, dx \leq \varepsilon CM_1 + s_1^{2_\alpha^*} \int_{A_n \cap \complement B(0,r)}
 c(x)\, dx < (CM_1+1)\varepsilon
\]
for all $n \in \mathbb{N}$. From the compact embedding of $L^{\alpha,2}(\mathbb{R}^N)$ into 
$L_{\mathrm{loc}}^q(\mathbb{R}^N)$ and the boundedness of $c$, we know that
\[
 \lim_{n \to +\infty} \int_{B(0,r)}c(x) |v_n(x)|^q \, dx = 0.
\]
Hence $v_n \to 0$ strongly in $L^q(c\, d\mathcal{L}^N)$. The proof is complete.
\end{proof}
\begin{remark}
 If assumption (c2) is replaced by $\lim_{|x| \to +\infty} c(x)=0$, then the compactness of the embedding is
 classical: see for instance \cite{Stuart}*{Lemma 3.2}. A different assumption ensuring the compactness of the weighted embedding appears in \cite{Chabrowski}.
\end{remark}

\begin{proposition} \label{prop:2}
 If $v_n \to v$ weakly in $L^{\alpha,2}(\mathbb{R}^N)$ then
 \begin{equation} \label{eq:3}
  \lim_{n \to +\infty} \int_{\mathbb{R}^N} c(x)F(v_n(x))\, dx = \int_{\mathbb{R}^N} c(x)F(v(x))\, dx
 \end{equation}
and
\begin{equation} \label{eq:4}
 \lim_{n \to +\infty} \int_{\mathbb{R}^N} c(x)f(v_n(x))v_n(x)\, dx = \int_{\mathbb{R}^N} c(x)f(v(x))v(x)\, dx
\end{equation}
\end{proposition}
\begin{proof}
 Let us fix $q \in (2,2_\alpha^*)$ and $\varepsilon>0$. From (f1)--(f3) there exists a constant $C>0$ such that
 \begin{equation} \label{eq:5}
  |c(x)F(s)| \leq \varepsilon C \left( s^2+|s|^{2_\alpha^*} \right) + c(x)|s|^q
 \end{equation}
for every $s \in \mathbb{R}$ and every $x \in \mathbb{R}^N$. Proposition \ref{prop:1} tells us that
\[
 \lim_{n \to +\infty} \int_{\mathbb{R}^N} c(x)|v_n(x)|^q \, dx = \int_{\mathbb{R}^N} c(x)|v(x)|^q \, dx,
\]
so that for some $r>0$,
\[
 \int_{\complement B(0,r)} c(x) |v_n(x)|^q \, dx < \varepsilon \quad\text{for all $n \in \mathbb{N}$}.
\]
But $\{v_n\}_n$ is bounded in $L^{\alpha,2}(\mathbb{R}^N)$ and therefore also in $L^2(\mathbb{R}^N)$ and 
in $L^{2_\alpha^*}(\mathbb{R}^N)$.
We choose $M_2>0$ such that $\int_{\mathbb{R}^N} v_n(x)^2 \, dx \leq M_2$ and $\int_{\mathbb{R}^N} |v_n(x)|^{2_\alpha^*} \, dx \leq M_2$ for all $n \in \mathbb{N}$.
Integrating (\ref{eq:5}) we get
\[
 \left| \int_{\complement B(0,r)} c(x)F(v_n(x))\, dx \right| < \left( 2CM_2+1 \right)\varepsilon \quad \text{for all $n \in \mathbb{N}$}.
\]
Applying a general convergence result (see for instance \cite{Chang}*{Lemma 2.4}) to the sequence $\{\|v_n\|_{L^{2_\alpha^*}(B(0,r))}\}_n$ together with our assumption
(f2) we conclude that
\[
 \lim_{n \to +\infty} \int_{B(0,r)} c(x)F(v_n(x))\, dx = \int_{B(0,r)} c(x)F(v(x))\, dx.
\]
This shows the validity of (\ref{eq:3}), and a similar argument proves also (\ref{eq:4}).
\end{proof}
\begin{proposition} \label{prop:3}
 The sequence $\{u_n\}_n$ introduced in (\ref{eq:6}) is bounded in $L^{\alpha,2}(\mathbb{R}^N)$.
\end{proposition}
\begin{proof}
 For every $n \in \mathbb{N}$, let $t_n \in [0,1]$ be chosen so that $J(t_nu_n) = \max_{0 \leq t \leq 1} J(t u_n)$. Let us prove that
 the sequence $\{J(t_n u_n)\}_n$ is bounded from above in $\mathbb{R}$. The conclusion is trivial if either $t_n=0$ or $t_n=1$.
 If $0<t_n<1$, then $DJ(t_n u_n)u_n=0$. As a consequence,
 \[
  2J(t_n u_n) = 2 J(t_n u_n) - DJ(t_n u_n)u_n = \int_{\mathbb{R}^N} c(x)H(t_n u_n(x))\, dx,
 \]
for $H(s) = s f(s) - 2F(s)$. 
Since $u_n \geq 0$ and $H$ is non-decreasing,
\[
 2 J(t_n u_n) \leq \int_{\mathbb{R}^N} c(x) H(u_n(x))\, dx = 2J(u_n)-DJ(u_n)u_n = 2J(u_n)+o(1).
\]
This shows that $\{J(t_n u_n)\}_n$ is bounded from above. To complete the proof, we argue by contradiction. Let us
assume that (possibily along a subsequence) $\lim_{n \to +\infty} \|(I-\Delta)^{\alpha/2}u_n \|_2 =+\infty$. We normalize
$u_n$ by introducing $w_n = u_n / \|(I-\Delta)^{\alpha/2}u_n \|_2$. Without loss of generality, we may assume that
$w_n \to w$ weakly in $L^{\alpha,2}(\mathbb{R}^N)$. Divinding out the relation $J(u_n) \to c$, we find
\[
 o(1)+\frac{1}{2} = \int_{\mathbb{R}^N} \frac{c(x)F(u_n(x))}{ \left\|(I-\Delta)^{\alpha/2}u_n \right\|_2^2}\, dx = \int_{\mathbb{R}^N} \frac{c(x)F(u_n(x))}{|u_n(x)|^2}w_n(x)^2 \, dx.
\]
By (f3), to each $T>0$ we can attach $\xi>0$ with the property that $|s| \geq \xi$ implies $F(s) \geq Ts^2$. Therefore
\[
 o(1)+\frac{1}{2}\geq \int_{\complement w^{-1}(0) \cap \{|u_n| \geq \xi \}} \frac{c(x)F(u_n(x))}{|u_n(x)|^2}w_n(x)^2 \, dx \geq
 T \int_{\complement w^{-1}(0) \cap \{|u_n| \geq \xi \}} c(x) w_n(x)^2 \, dx.
\]
An application of Fatou's lemma yields $\frac{1}{2} \geq T \int_{\complement w^{-1}(0)} c(x) w(x)^2 \, dx$ and consequently $w=0$ by letting $T \to +\infty$.

Finally, given $T>0$, we have
\[
 J(t_n u_n) \geq J \left( \frac{T}{\left\|(I-\Delta)^{\alpha/2}u_n \right\|_2} u_n \right) = J(T w_n) = \frac{T^2}{2}
 -\int_{\mathbb{R}^N} c(x) F(Tw_n(x))\, dx.
\]
Proposition \ref{prop:2} yields $\lim_{n \to +\infty} \int_{\mathbb{R}^N} c(x) F(T w_n(x))\, dx =0$ and thus
\[
 \liminf_{n \to +\infty} J(t_n u_n) \geq \frac{T^2}{2}.
\]
Letting $T \to +\infty$ we contradict the boundedness from above of the sequence $\{J(t_n u_n)\}_n$. The proof is complete.
\end{proof}
Let us state and prove the main result of this section.
\begin{theorem}
 Suppose that (f1)--(f3) and (c1)--(c2) hold. Then equation (\ref{eq:2}) possesses at least one positive solution. Moreover $u \in C^{0,\mu}(\mathbb{R}^N)$ for some $q_0 \in [2,+\infty)$ and $\mu \in (0,1)$.
 \end{theorem}
\begin{proof}
 We first prove that our Cerami sequence $\{u_n\}_n$ is relatively compact in $L^{\alpha,2}(\mathbb{R}^N)$.
 To this aim, by virtue of Proposition \ref{prop:3}, we may assume that, up to a subsequence, $u_n \to u$ weakly in $L^{\alpha,2}(\mathbb{R}^N)$. 
 Since $\{u_n\}_n$ is a Cerami sequence, we find
 \[
  \lim_{n \to +\infty} \left\|(I-\Delta)^{\alpha/2} \right\|_2^2 = \lim_{n \to +\infty} \int_{\mathbb{R}^N} c(x) f(u_n(x))u_n(x)\, dx.
 \]
By Proposition \ref{prop:2},
\[
 \lim_{n \to +\infty} \int_{\mathbb{R}^N} c(x)f(u_n(x))u_n(x)\, dx = \int_{\mathbb{R}^N} c(x) f(u(x))u(x)\, dx.
\]
Exploiting again the fact that $DJ(u_n)u \to 0$, we have 
\[
\left\|(I-\Delta)^{\alpha/2} u \right\|_2^2 = \int_{\mathbb{R}^N} c(x)f(u(x))u(x)\, dx.
\]
In particular $\lim_{n \to +\infty} \|(I-\Delta)^{\alpha/2} u_n\|_2^2
= \|(I-\Delta)^{\alpha/2} u\|_2^2$, or $u_n \to u$ strongly in
$L^{\alpha,2}(\mathbb{R}^N)$.  But $J(u)=c_{\mathrm{mp}}$ and
$DJ(u)=0$, so that $u$ weakly solves (\ref{eq:2}). The positivity of
$u$ follows from the fact that $u_n \geq 0$ and the positivity of the
Bessel function $G_\alpha$ see~\cite{FelmerVergara}*{Proposition 3.2}.
The regularity of $u$ follows with minor changes from the arguments developed in \cite{Felmer}.
\end{proof}
\begin{remark}
We have been sketchy about the regularity theory for our solutions, since the Bessel operator is precisely the main tool to develop a regularity theory for the fractional Laplacian, see \cite{Felmer}*{Appendix A}. In this sense, the fractional Laplacian is \emph{harder} to analyze.
Moreover, by similar arguments as those in \cite{FelmerVergara} it can be shown that our solution decays exponentially fast at infinity. This is a common feature for local elliptic partial differential operators, while it is false for the fractional Laplacian, see \cite{Felmer}.
\end{remark}

\section{Potentials having a finite limit} \label{sec:4}

The main tool that we used to solve equation (\ref{eq:2}) is the compactness of the embedding $L^{\alpha,2}(\mathbb{R}^N) \to
L^q(c\, d\mathcal{L}^N)$ stated in Proposition \ref{prop:1}. In this section we study a model case in which a different approach must
be used. We consider
\begin{equation} \label{eq:7}
 (I-\Delta)^\alpha u = \lambda b(x) |u|^{p-2}u + c(x) |u|^{q-2}u \quad\text{in $\mathbb{R}^N$},
\end{equation}
where $p$, $q \in (2,2_\alpha^*)$ and $\lambda>0$ is a parameter. We will assume that $c>0$ satisfies the condition (K), and impose the 
following ones on the potential $b$:
\begin{itemize}
 \item[(b1)] $b \in L^\infty(\mathbb{R}^N)$, $b \geq 0$ but not identically zero, and $\lim_{|x| \to +\infty} b(x) = \bar{b}$.
\end{itemize}
Weak solutions to equation (\ref{eq:7}) correspond to critical points of the Euler functional $J_\lambda \colon L^{\alpha,2}(\mathbb{R}^N)
\to \mathbb{R}$ defined by
\[
 J_\lambda(u)=\frac{1}{2}\left\| (I-\Delta)^{\alpha/2} u\right\|_2^2 - \frac{\lambda}{p} \int_{\mathbb{R}^N} b(x) |u(x)|^p\, dx 
 - \frac{1}{q} \int_{\mathbb{R}^N} c(x) |u(x)|^q \, dx.
\]
\begin{remark} \label{rem:1}
 The functional $u \mapsto \int c(x) |u(x)|^q \, dx$ is weakly sequentially continuous by the results of the previous section.
\end{remark}
Let us introduce the artificial constraint
\[
 \mathcal{M}_\lambda = \left\{ u \in L^{\alpha,2}(\mathbb{R}^N) \setminus \{0\} \mid DJ_\lambda(u)u=0 \right\},
\]
and standard arguments show that $\mathcal{M}_\lambda$ is a natural constraint for $J_\lambda$. In particular,
any solution of the minimization problem
\begin{equation*} 
 I_\lambda = \inf_{u \in \mathcal{M}_\lambda} J_\lambda (u)
\end{equation*}
is a solution to equation (\ref{eq:7}).

For $\lambda=0$, we introduce the constraint
\[
 \mathcal{M}_0 = \left\{ u \in L^{\alpha,2}(\mathbb{R}^N) \setminus \{0\} \mid DJ_0(u)u=0 \right\}
\]
corresponding to the Euler functional
\[
 J_0(u)=\frac{1}{2} \left\| (I-\Delta)^{\alpha/2} u\right\|_2^2 - \frac{1}{q} \int_{\mathbb{R}^N} c(x) |u(x)|^q \, dx.
\]
By Remark \ref{rem:1},  the minimization problem
\[
 I_0 = \inf_{u \in \mathcal{M}_0} J_0(u)
\]
is solved by some function $u_0 \in L^{\alpha,2}(\mathbb{R}^N)$ that satisfies $(I-\Delta)^{\alpha} u_0 = c(x)|u_0|^{q-2}u_0$. 
\begin{lemma}
There results $I_\lambda \leq I_0$ for all $\lambda>0$.
\end{lemma}
\begin{proof}
Let $u \in L^{\alpha,2}(\mathbb{R}^N)$ be such that
\[
\left\| (I-\Delta)^{\alpha/2}u \right\|_2^2  = \int_{\mathbb{R}^N} c(x) |u(x)|^{q} \, dx.
\]
Pick $\bar{\sigma} \in (0,1)$ such that $v=\bar{\sigma}u \in \mathcal{M}_\lambda$. If we differentiate
\[
h(\sigma) = \frac{\sigma^2}{2} \left\| (I-\Delta)^{\alpha/2}u \right\|_2^2  - \frac{\sigma^q}{q} \int_{\mathbb{R}^N} c(x) |u(x)|^q \, dx
\]
and remark that $Dh(\sigma)>0$ for every $\sigma \in (0,1)$, we may conclude that 
\begin{multline*}
I_\lambda(v) = \frac{\bar{\sigma}}{2} \left\| (I-\Delta)^{\alpha/2}u \right\|_2^2 - \frac{\bar{\sigma}^p}{p} \int_{\mathbb{R}^N} \lambda b(x) |u(x)|^p \, dx - \frac{\bar{\sigma}^q}{q} \int_{\mathbb{R}^N} c(x) |u(x)|^q \, dx \\
< \frac{\bar{\sigma}}{2} \left\| (I-\Delta)^{\alpha/2}u \right\|_2^2 - \frac{\bar{\sigma}^q}{q} \int_{\mathbb{R}^N} c(x) |u|^q \, dx \\
< \frac{1}{2} \left\| (I-\Delta)^{\alpha/2}u \right\|_2^2 - \frac{1}{q} \int_{\mathbb{R}^N} c(x) |u|^q \, dx = I_0.
\end{multline*}
\end{proof}
Let us introduce
\[
 \alpha_1 = \inf_{\|u\|_p=1} \left\| (I-\Delta)^{\alpha/2} u \right\|_2^2>0.
\]
The main result of this section reads as follows.
\begin{theorem}
 Under our assumptions on $b$ and $c$, if
 \begin{equation} \label{eq:13}
  I_\lambda < \frac{p-2}{2p} \alpha_1^{\frac{p}{p-2}} \left( \lambda \bar{b} \right)^{\frac{2}{2-p}},
 \end{equation}
 then equation (\ref{eq:7}) possesses a nontrivial solution.
\end{theorem}
\begin{proof}
We follow the ideas developed in \cite{Chabrowski}. By Ekeland's variational principle, there exists a minimizing sequence $\{u_n\}_n \subset \mathcal{M}_\lambda$ such
 that $\lim_{n \to +\infty} J_\lambda(u_n)=I_\lambda$ and $\lim_{n \to +\infty} DJ_\lambda(u_n)=0$. It is readily seen that 
 the sequence $\{u_n\}_n$ is bounded in $L^{\alpha,2}(\mathbb{R}^N)$, so that we may assume without loss of generality that
 $u_n \to u$ weakly in $L^{\alpha,2}(\mathbb{R}^N)$. Let us set
 \begin{align*}
  \alpha_\infty &= \lim_{R \to +\infty} \limsup_{n \to +\infty} \int_{\complement B(0,R)} |u_n(x)|^p \, dx \\
  \beta_\infty &= \lim_{R \to +\infty} \limsup_{n \to +\infty} \int_{\complement B(0,R)} |(I-\Delta)^{\alpha/2} u_n|^2 \, dx.
 \end{align*}
We claim that
\begin{align}
&\alpha_1 \alpha_\infty^{2/p} \leq \beta_\infty  \label{eq:9}\\
&\limsup_{n \to +\infty} \int_{\mathbb{R}^N} |u_n(x)|^p \, dx = \int_{\mathbb{R}^N} |u(x)|^p\, dx + \alpha_\infty \label{eq:10}\\
&\limsup_{n \to +\infty} \int_{\mathbb{R}^N} |(I-\Delta)^{\alpha/2}u_n|^2 \, dx \geq \int_{\mathbb{R}^N} |(I-\Delta)^{\alpha/2} u|^2 \, dx + \beta_\infty.\label{eq:11}	
\end{align}
Let us fix a smooth, positive cutoff function $\varphi$ that equals one on a neighborhood of the origin. Define $\varphi_R = \varphi(\cdot / R)$ and $\tilde{\varphi}_R = 1-\varphi_R$. By definition of $\alpha_1$, we have
\[
  \alpha_1 \|\tilde{\varphi}_R u_n \|_p^2 \leq \| (I-\Delta)(\tilde{\varphi}_R u_n)\|_2^2
\]
By Lemmas \ref{lem:1} and \ref{lem:2} in the Appendix, 
\begin{multline*}
	(I-\Delta)^{\alpha/2} (\tilde{\varphi}_R u_n) = 
	(I-\Delta)^{\alpha/2} u_n - (I-\Delta)^{\alpha/2}(\varphi_R u_n) = \\
	(I-\Delta)^{\alpha/2}u_n -\varphi_R (I-\Delta)^{\alpha/2}u_n + o_n(1) =
	(1-\varphi_R)(I-\Delta)^{\alpha/2} u_n +o_n(1)
\end{multline*}
where $o_n(1) \to 0$ as $n \to +\infty$ in $L^2(\mathbb{R}^N)$. Therefore
\[
\alpha_1 \|\tilde{\varphi}_R u_n \|_p^p \leq \int_{\mathbb{R}^N} |\tilde{\varphi}_R|^2 |(I-\Delta)^{\alpha/2}u_n|^2 \, dx + o_n(1)
\]
Letting $n \to +\infty$ and $R \rightarrow +\infty$, we get (\ref{eq:9}). The relations (\ref{eq:10}) and (\ref{eq:11}) are easy and we omit their proofs.

To complete the proof, we need to show that $\alpha_\infty=0$. We argue by contradiction, and suppose that $\alpha_\infty>0$. Let us consider once more the function $\tilde{\varphi}_R$. Since $DJ(u_n)(\tilde{\varphi}_Ru_n) \to 0$ as $n \to +\infty$, we find that
\[
\beta_\infty \leq \lambda \bar{b} \alpha_\infty.
\]
Combining with (\ref{eq:9}) yields
\begin{equation} \label{eq:12}
\alpha_\infty \geq \left( \frac{\alpha_1}{\lambda \bar{b}}
\right)^{\frac{p}{p-2}}.
\end{equation}
Moreover,
\begin{multline*}
	J(u_n)-\frac{1}{2} DJ(u_n)u_n = \left( \frac{1}{2}-\frac{1}{p} \right) \lambda \int_{\mathbb{R}^N} b(x) |u_n(x)|^p\, dx + \left( \frac{1}{2}-\frac{1}{p} \right) \lambda \int_{\mathbb{R}^N} c(x) |u_n(x)|^q \, dx \\
	\geq \frac{(p-2)\lambda}{2p} \int_{\mathbb{R}^N} b(x) |u_n(x)|^p \tilde{\varphi}_R(x) \, dx,
\end{multline*}
and letting first $n \rightarrow +\infty$ and then $R \rightarrow +\infty$ we have that
\[
I_\lambda \geq \frac{p-2}{2p} \lambda \bar{b}\alpha_\infty.
\]
Together with (\ref{eq:12}) we get the contradiction
\[
I_\lambda \geq \frac{p-2}{2p} \left( \frac{\alpha_1}{\lambda \bar{b}} \right)^{\frac{p}{p-2}} \bar{b}\lambda = \frac{p-2}{2p} \alpha_1^{\frac{p}{p-2}} \left( \lambda \bar{b} \right)^{\frac{2}{2-p}}.
\]
\end{proof}
\begin{remark}
The lack of precise information about solutions to the limit equation 
\begin{displaymath}
(I-\Delta)^\alpha u = \bar{b} |u|^{p-2}u \quad\text{in $\mathbb{R}^N$},	
\end{displaymath}
and particularly the lack of a \emph{uniqueness} result does not allow us to state the assumption (\ref{eq:13}) in terms of the ground state energy of the associated Euler functional.
\end{remark}

\section{Concave-convex nonlinearities and sign-changing potentials} \label{sec:5}

In this section, following \cites{Brown,Goyal},  we consider the equation
\begin{equation} \label{eq:14}
(I-\Delta)^{\alpha} u = \lambda b(x)|u|^{p-2}u + c(x) |u|^{q-2}u \quad\text{in $\mathbb{R}^N$}
\end{equation}
where $N>2\alpha$, \(\lambda>0\) and  and $1<p<2<q<\frac{2N}{N-2\alpha}$. 
We assume that both $|b|$ and $|c|$  satisfy the compactness conditions (K). 
\begin{remark} \label{rem:2}
If $u_n \to 0$ weakly in $L^{\alpha,2}(\mathbb{R}^N)$, it follows from Proposition \ref{prop:1} that
\begin{equation*}
\left| \int_{\mathbb{R}^N} b(x) |u_n(x)|^p \, dx \right| \leq \int_{\mathbb{R}^N} |b(x)| |u_n(x)|^p \, dx \rightarrow 0
\end{equation*}
as $n \to +\infty$ along a subsequence. Therefore $\lim_{n \to +\infty} \int_{\mathbb{R}^N} b(x) |u_n(x)|^p \, dx = 0$. 
By the same token, $\lim_{n \to +\infty} \int_{\mathbb{R}^N} c(x) |u_n(x)|^q \, dx = 0$.
\end{remark}
\begin{definition} \label{def:2}
We say that a continuous function $w$ changes sign if both sets $\{x\in \mathbb{R}^N \mid w(x)>0 \}$ and  $\{x\in \mathbb{R}^N \mid w(x)<0 \}$
are nonempty.
\end{definition}

Let us recall from \cite{Fall} that (\ref{eq:14}) is equivalent to the Neumann system
\begin{equation} \label{eq:15}
\left\{
\begin{array}{ll}
-\operatorname{div} \left( y^{1-2\alpha}\nabla v \right)+y^{1-2\alpha} v =0&\text{in \( \mathbb{R}^{N+1}_+ \)}\\
- y^{1-2\alpha} \frac{\partial v}{\partial y} = \kappa_\alpha \left\{ \lambda b(x) |u(x,0)|^{p-2}u(x,0)+c(x) |u(x,0)|^{q-2}u(x,0) \right\} &\text{on \(  \mathbb{R}^N \times \{0\} \)}
\end{array}
\right.
\end{equation}
Weak solutions to (\ref{eq:15}) correspond to critical points of the Euler functional $E_\lambda \colon H^1(\mathbb{R}_{+}^{N+1},y^{1-2\alpha}) \to \mathbb{R}$ defined by
\begin{multline*} 
E_\lambda (v) = \frac{1}{2} \int_{\mathbb{R}_{+}^{N+1}} y^{1-2\alpha} \left( |\nabla v(x,y)|^2 + v(x,y)^2 \right) \, dx\, dy \\
{}- \frac{\lambda}{p} \int_{\mathbb{R}^N} b(x) |\tr v(x)|^{p} \, dx - \frac{1}{q} \int_{\mathbb{R}^N} c(x) |\tr v(x)|^{q}\, dx,
\end{multline*}
where $H^1(\mathbb{R}_{+}^{N+1},y^{1-2\alpha})$ is the completion of
$C_0^\infty(\mathbb{R}_{+}^{N+1})$ with respect to the weighted norm
\[
\|v\|_{H^1(\mathbb{R}_{+}^{N+1},y^{1-2\alpha})}=\left(
\int_{\mathbb{R}_{+}^{N+1}} y^{1-2\alpha} \left( |\nabla v(x,y)|^2 +
v(x,y)^2 \right)\, dx\, dy \right)^{1/2}
\]
and $\tr \colon H^{1}(\mathbb{R}_{+}^{N+1},y^{1-2\alpha}) \to
L^{\alpha,2}(\mathbb{R}^N)$ is the continuous trace operator defined
in \cite{Fall}*{Proposition 6.2}. In particular
\begin{equation} \label{eq:18}
  \kappa_\alpha \|\tr v\|^2_{L^{\alpha,2}(\mathbb{R}^N)} \leq
  \|v\|_{H^1(\mathbb{R}_{+}^{N+1})}^2 \quad \text{for all $v \in
    H^1(\mathbb{R}_{+}^{N+1})$}.
\end{equation}
\begin{remark}
Up to a rescaling, we can (and will) assume that $\kappa_\alpha =1$.
\end{remark}

The Nehari manifold associated to (\ref{eq:15}) is
\begin{equation*}
\mathcal{N}_\lambda = \left\{ v \in
H^1(\mathbb{R}_{+}^{N+1},y^{1-2\alpha}) \setminus \{0\} \mid
DE_\lambda (v)v=0 \right\}
\end{equation*}
or the set of those $v \in H^1(\mathbb{R}_{+}^{N+1},y^{1-2\alpha})
\setminus \{0\}$ satisfying
\begin{equation*}
\int_{\mathbb{R}_{+}^{N+1}} y^{1-2\alpha} \left( |\nabla v(x,y)|^2 +
v(x,y)^2 \right) \, dx\, dy = \lambda \int_{\mathbb{R}^N} b(x) |\tr
v(x)|^{p} \, dx + \int_{\mathbb{R}^N} c(x) |\tr v(x)|^{q}\, dx.
\end{equation*}
To each $v \in H^1(\mathbb{R}_{+}^{N+1},y^{1-2\alpha})$ we attach its
fiber map $\varphi_v \colon [0,+\infty) \to \mathbb{R}$ defined by
  $\varphi_v(t)=E_\lambda (tv)$. By a direct calculation,
\begin{multline*}
D\varphi_v (t) = t \int_{\mathbb{R}_{+}^{N+1}} y^{1-2\alpha} \left(
|\nabla v(x,y)|^2 + v(x,y)^2 \right) \, dx\, dy - \lambda t^{p-1}
\int_{\mathbb{R}^N} b(x) |\tr v(x)|^{p} \, dx \\ {}- t^{q-1}
\int_{\mathbb{R}^N} c(x) |\tr v(x)|^{q}\, dx
\end{multline*}
and
\begin{multline*}
D^2 \varphi_v (t) = \int_{\mathbb{R}_{+}^{N+1}} y^{1-2\alpha} \left(
|\nabla v(x,y)|^2 + v(x,y)^2 \right) \, dx\, dy - (p-1) \lambda t^{p-2}
\int_{\mathbb{R}^N} b(x) |\tr v(x)|^{p} \, dx \\ {}-(r-1) t^{q-2}
\int_{\mathbb{R}^N} c(x) |\tr v(x)|^{q}\, dx
\end{multline*}
\begin{remark}
Clearly, $tv \in \mathcal{N}_\lambda$ if and only if $D\varphi_v (t)=0$. In particular, $v \in \mathcal{N}_\lambda$ if and only if $D \varphi_v (1)=0$.
\end{remark}
We decompose
\[
\mathcal{N}_\lambda = \mathcal{N}_\lambda^{-} \cup \mathcal{N}_\lambda^0 \cup \mathcal{N}_\lambda^{+},
\]
where
\begin{align*}
\mathcal{N}_\lambda^{-} &= \left\{ v \in \mathcal{N}_\lambda \mid D^2 \varphi_v (1) <0 \right\} \\
\mathcal{N}_\lambda^{0} &= \left\{ v \in \mathcal{N}_\lambda \mid D^2 \varphi_v (1) =0 \right\} \\
\mathcal{N}_\lambda^{+} &= \left\{ v \in \mathcal{N}_\lambda \mid D^2 \varphi_v (1) >0 \right\}.
\end{align*}
Let us define
\begin{align*}
C^{+} &= \left\{ v \in H^1(\mathbb{R}_{+}^{N+1},y^{1-2\alpha}) \mid \int_{\mathbb{R}^N} c(x) |\tr v(x)|^{q} \, dx > 0 \right\} \\
C^{-} &= \left\{ v \in H^1(\mathbb{R}_{+}^{N+1},y^{1-2\alpha}) \mid \int_{\mathbb{R}^N} c(x) |\tr v(x)|^{q} \, dx < 0 \right\} \\
C^0 &= \left\{ v \in H^1(\mathbb{R}_{+}^{N+1},y^{1-2\alpha}) \mid \int_{\mathbb{R}^N} c(x) |\tr v(x)|^{q}\, dx =0 \right\}\\
B^{+} &= \left\{ v \in H^1(\mathbb{R}_{+}^{N+1},y^{1-2\alpha}) \mid \int_{\mathbb{R}^N} b(x) |\tr v(x)|^{p} \, dx > 0 \right\} \\
B^{-} &= \left\{ v \in H^1(\mathbb{R}_{+}^{N+1},y^{1-2\alpha}) \mid \int_{\mathbb{R}^N} b(x) |\tr v(x)|^{p} \, dx < 0 \right\} \\
B^{0} &= \left\{ v \in H^1(\mathbb{R}_{+}^{N+1},y^{1-2\alpha}) \mid \int_{\mathbb{R}^N} b(x) |\tr v(x)|^{p} \, dx = 0 \right\}.
\end{align*}
\begin{lemma} \label{lem:4}
\begin{enumerate}
\item If $v \in B^{-} \cap C^{-}$, then no multiple of $v$ belongs to $\mathcal{N}_\lambda$.
\item If either $v \in B^{+} \cap C^{-}$ or $v \in B^{-} \cap C^{+}$, then there exists one and only one $t(v)>0$ such that $t(v) v \in \mathcal{N}_\lambda$.
\end{enumerate}
\end{lemma}
\begin{proof} 
Both statements can be proved by an elementary inspection of the fibering map $\varphi_v$.
\end{proof}
The case $v \in B^{+} \cap C^{+}$ requires some additional care.
\begin{lemma}
There exists $\lambda_0>0$ such that $\lambda < \lambda_0$ implies $\varphi_v>0$ for all $v \in H^1(\mathbb{R}_{+}^{N+1},y^{1-2\alpha})$. If $\lambda<\lambda_0$ and $v \in B^{+} \cap C^{+}$, then $\varphi_v$ possesses exactly two critical points.
\end{lemma}
\begin{proof}
Let $v \in H^1(\mathbb{R}_{+}^{N+1},y^{1-2\alpha})$ be such that $\int_{\mathbb{R}^N} c(x) |\tr v(x)|^{q} \, dx >0$. We introduce 
\begin{equation*}
F_v (t) = \frac{t^2}{2} \int_{\mathbb{R}_{+}^{N+1}} y^{1-2\alpha} \left( |\nabla v(x,y)|^2 + v(x,y)^2 \right)\, dx\, dy 
- \frac{t^{q}}{q} \int_{\mathbb{R}^N} c(x) | \tr v(x)|^{q} \, dx.
\end{equation*}
Since
\begin{equation*}
D F_v (t) = t \int_{\mathbb{R}_{+}^{N+1}} y^{1-2\alpha} \left( |\nabla v(x,y)|^2 + v(x,y)^2 \right)\, dx\, dy 
- t^{q-1}\int_{\mathbb{R}^N} c(x) | \tr v(x)|^{q} \, dx,
\end{equation*}
the function $F_v$ attains its maximum at
\[
t^* = \left( \frac{\int_{\mathbb{R}^{N+1}_{+}} y^{1-2\alpha} \left( |\nabla (x,y)|^2 + v(x,y)^2 \right)\, dx \, dy}{\int_{\mathbb{R}^N} c(x) |\tr v(x)|^{q} \, dx} \right)^{\frac{1}{q-2}}.
\]
Moreover
\begin{align*}
F_v(t^*) &= \left( \frac{1}{2} - \frac{1}{q} \right) \left(
\frac{\left(\int_{\mathbb{R}^{N+1}_{+}} y^{1-2\alpha} \left( |\nabla (x,y)|^2 + v(x,y)^2 \right)\, dx \, dy \right)^{q}}{\left( \int_{\mathbb{R}^N} c(x) |\tr v(x)|^{q} \, dx \right)^2}
\right)^{\frac{1}{q-2}} \\
D^2 F_v(t^*) &= \left(1-q \right) \int_{\mathbb{R}_{+}^{N+1}} y^{1-2\alpha} \left( |\nabla v(x,y)|^2 + v(x,y)^2 \right) \, dx \, dy <0.
\end{align*}
Let $S_{q}$ be the best constant for the inequality
\[
\left(\int_{\mathbb{R}^N} |\tr v(x)|^{q}\, dx \right)^{\frac{1}{q}} \leq S_{q} \sqrt{\int_{\mathbb{R}_{+}^{N+1}} y^{1-2\alpha} \left( |\nabla v(x,y)|^2 + v(x,y)^2 \right) \, dx\, dy}
\]
that follows from (\ref{eq:18}) and Theorem \ref{th:1}.
Then 
\begin{equation*} 
F_v (t^*) \geq \left( \frac{1}{2} - \frac{1}{q} \right) \left( \frac{1}{\|b^{+}\|_\infty S_{q}^{2r}} \right)^{\frac{1}{q-2}} = \delta >0
\end{equation*}
with $\delta$ independent of $v$. Moreover,
\begin{multline*}
\frac{(t^*)^{p}}{p} \int_{\mathbb{R}^N} b(x) |\tr v(x)|^{p} \, dx \\
\leq \frac{\|b\|_\infty}{p} S_{p}^{p} \left( \frac{\int_{\mathbb{R}_{+}^{N+1}} y^{1-2\alpha} \left( |\nabla v(x,y)|^2 + v(x,y)^2 \right)\, dx\, dy}{\int_{\mathbb{R}^N} c(x) |\tr v(x)|^{q}\, dx} \right)^{\frac{p}{q-2}} \|v\|_{H^1(\mathbb{R}^{N+1}_{+},y^{1-2\alpha})}^{p} \\
=\frac{\|b\|_\infty}{p} S_{p}^{p}  \left( \frac{\left( \int_{\mathbb{R}_{+}^{N+1}} y^{1-2\alpha} \left( |\nabla v(x,y)|^2 + v(x,y)^2 \right) dx \, dy \right)^{q}}{\left( \int_{\mathbb{R}^N} c(x) |\tr v(x)|^{q} \, dx \right)^2} \right)^{\frac{p}{2(q-2)}}  \\
=\frac{\|b\|_\infty}{p} S_{p}^{p}  \left( \frac{2q}{q-2} \right)^{\frac{p}{2}} F_v(t^*)^{\frac{p}{2}} = c F_v(t^*)^{\frac{p}{2}} .
\end{multline*}
Therefore
\begin{equation*}
\varphi_v(t^*) \geq F_v(t^*) - \lambda c F_v (t^*)^{\frac{p}{2}} 
= F_v (t^*)^{\frac{p}{2}} \left( F_v(t^*)^{\frac{2-p}{2}} - \lambda c \right) \geq \delta^{\frac{p}{2}} \left( \delta^{\frac{2-p}{2}} - \lambda c \right).
\end{equation*}
We complete the proof by choosing \(\lambda < c^{-1} \delta^{\frac{2-p}{2}} = \lambda_0\).
\end{proof}
\begin{corollary} \label{cor:2}
Let $\lambda<\lambda_0$. There exists $\delta_1>0$ such that $E_\lambda (v) \geq \delta_1$ for every $v \in \mathcal{N}_\lambda^{-}$.
\end{corollary}
\begin{proof}
Indeed, if $v \in \mathcal{N}_\lambda^{-}$, then $\varphi_v$ has a positive global maximum at $t=1$, and in addition 
\[
\int_{\mathbb{R}^N} b(x) |\tr v(x)|^{p} \, dx >0. 
\]
Hence
\begin{equation*}
E_\lambda(v) = \varphi_v(1) = \varphi_v(t^*) \geq F_v(t^*)^{\frac{p-2}{2}} \left(F_v(t^*)^{\frac{2-p}{2}}-\lambda c \right) \geq \delta^{\frac{p}{2}} \left( \delta^{\frac{2-p}{2}}-\lambda c \right)>0
\end{equation*}
provided that $\lambda < \lambda_0$.
\end{proof}
\begin{corollary}
If $\lambda \in (0,\lambda_0)$, then $\mathcal{N}_\lambda^0 = \emptyset$.
\end{corollary}
\begin{lemma}
Let $v$ be a local minimum point of $E_\lambda$ on either $\mathcal{N}_\lambda^{+}$ or $\mathcal{N}_\lambda^{-}$ such that $v \notin \mathcal{N}_\lambda^0$. Then $DE_\lambda(v)=0$.
\end{lemma}
\begin{proof}
By the Lagrange Multiplier Rule, $DE_\lambda (v) = \mu DI_\lambda (v)$ for some $\mu \in \mathbb{R}$, where we have set $I_\lambda(v) = DE_\lambda (v)v$. Hence $DE_\lambda (v)v = \mu DI_\lambda (v)v = \mu D^2 \varphi_v(1)=0$. Since $v \notin \mathcal{N}_\lambda^0$ we must have $\mu =0$.
\end{proof}
\begin{lemma} \label{lem:3}
The restriction of $E_\lambda$ to $\mathcal{N}$ is bounded from below and coercive.
\end{lemma}
\begin{proof}
Indeed, for any $v \in \mathcal{N}_\lambda$, we have
\begin{align*}
E_\lambda (v) &= \left( \frac{1}{2} - \frac{1}{q} \right) \|v\|_{H^1(\mathbb{R}^{N+1}_{+},y^{1-2\alpha})}^2 = \lambda \left( \frac{1}{p} - \frac{1}{q} \right) \int_{\mathbb{R}^N} b(x) |\tr v(x)|^{p} \, dx \\
&\geq c_1 \|v\|_{H^1(\mathbb{R}^{N+1}_{+},y^{1-2\alpha})}^2 - c_2 \|v\|_{H^1(\mathbb{R}^{N+1}_{+},y^{1-2\alpha})}^{p},
\end{align*}
and we conclude since $p<2$.
\end{proof}
\begin{lemma}
If $\lambda<\lambda_0$, then $E_\lambda$ attains its minimum on $\mathcal{N}_\lambda^{+}$.
\end{lemma}
\begin{proof}
We know from Lemma \ref{lem:3} that $E_\lambda$ is bounded from below on $\mathcal{N}_\lambda^{+}$. Therefore 
there exists a sequence $\{v_k\}_k \subset \mathcal{N}_\lambda^{+}$ with
\[
\lim_{k \to +\infty} E_\lambda (v_k) = \inf_{\mathcal{N}_\lambda^{+}} E_\lambda > -\infty.
\]
Again from Lemma \ref{lem:3} the sequence $\{v_k\}_k$ is bounded. We are then allowed to suppose, 
without loss of generality, that it converges weakly to some $v_\lambda$ in $
H^1(\mathbb{R}^{N+1}_{+},y^{1-2\alpha})$. Let us pick 
$v \in H^1(\mathbb{R}^{N+1}_{+},y^{1-2\alpha})$ such that 
$\int_{\mathbb{R}^N} b(x) |\tr v(x)|^{p} \, dx>0$; then by Lemma \ref{lem:4}  there 
exists $t_1=t_1(v)>0$ such that $t_1 v \in \mathcal{N}_\lambda^{+}$ and $E_\lambda (t_1 v)<0$. 
This implies that $\inf_{\mathcal{N}_\lambda^{+}}E_\lambda <0$. On $\mathcal{N}_\lambda$ we have
\[
E_\lambda (v_k) = \left( \frac{1}{2} - \frac{1}{q} \right) \|v_k\|^2_{H^1(\mathbb{R}_{+}^{N+1},y^{1-2\alpha})} - \lambda \left( \frac{1}{p} - \frac{1}{q} \right) \int_{\mathbb{R}^N} b(x) |\tr v_k(x)|^{p}\, dx
\]
and thus
\[
\lambda \left( \frac{1}{p} - \frac{1}{q} \right) \int_{\mathbb{R}^N} b(x) |\tr v_k(x)|^{p}\, dx = \left( \frac{1}{2} - \frac{1}{q} \right) \|v_k\|^2_{H^1(\mathbb{R}_{+}^{N+1},y^{1-2\alpha})} - E_\lambda (v_k).
\]
By Proposition \ref{prop:1}, Remark \ref{rem:2} and a semicontinuity argument, we let $k \to +\infty$ and get
\[
\lambda \left( \frac{1}{p} - \frac{1}{q} \right) \int_{\mathbb{R}^N} b(x) |\tr v(x)|^{p}\, dx  \geq \left(\frac{1}{2} - \frac{1}{q} \right) \|v\|^2_{H^1(\mathbb{R}_{+}^{N+1},y^{1-2\alpha})} - \inf_{\mathcal{N}_\lambda^{+}} E_\lambda > 0.
\]
We finally claim that \(v_k \to v_\lambda\) strongly. If not, then
\begin{equation*}
\int_{\mathbb{R}_{+}^{N+1}} y^{1-2\alpha} \left( |\nabla v(x,y)|^2 + v(x,y)^2 \right) \, dx \, dy 
< \liminf_{k \to +\infty} \int_{\mathbb{R}_{+}^{N+1}} y^{1-2\alpha} \left( |\nabla v_k(x,y)|^2 + v_k(x,y)^2 \right) \, dx \, dy.
\end{equation*}
Let us recall that 
\begin{multline*}
D\varphi_{v_k}(1) = t \int_{\mathbb{R}_{+}^{N+1}} y^{1-2\alpha} \left( |\nabla v_k(x,y)|^2 + v_k(x,y)^2 \right) 
\, dx\, dy \\ - \lambda t^{p-1} \int_{\mathbb{R}^N} b(x) | \tr v_k(x)|^{p} \, dx - t^{q-1} \int_{\mathbb{R}^N} c(x) 
|\tr v_k(x)|^{q}\, dx
\end{multline*}
and
\begin{multline*}
D\varphi_{v_\lambda}(1) = t \int_{\mathbb{R}_{+}^{N+1}} y^{1-2\alpha} \left( |\nabla v_\lambda(x,y)|^2 + 
v_\lambda(x,y)^2 \right) \, dx\, dy \\ 
- \lambda t^{p-1} \int_{\mathbb{R}^N} b(x) | \tr v_\lambda(x)|^{p} \, dx 
- t^{q-1} \int_{\mathbb{R}^N} c(x) |\tr v_\lambda(x)|^{q}\, dx.
\end{multline*}
Let $t_\lambda =t_\lambda(v_\lambda)>0$ be chosen in order that $t_\lambda v_\lambda \in \mathcal{N}_\lambda^{+}$: this is possible by Lemma \ref{lem:4}.
For $k \gg 1$, we have $D\varphi_{v_k} (t_\lambda)>0$, which yields $t_\lambda>1$. But then, by definition of $t_\lambda$,
\[
E_\lambda (t_\lambda v_\lambda) < E_\lambda (v_\lambda) < \lim_{k \to +\infty} E_\lambda (v_k) = \inf_{\mathcal{N}_\lambda^{+}} E_\lambda.
\]
This is impossible, and we must have $v_k \to v_\lambda$ strongly in $H^1(\mathbb{R}_{+}^{N+1},y^{1-2\alpha})$ as $k \to +\infty$. Since $\mathcal{N}_\lambda^0 = \emptyset$, there results $v_\lambda \in \mathcal{N}^{+}$, and the proof is complete.
\end{proof}
\begin{lemma}
If $\lambda < \lambda_0$, then $E_\lambda$ attains its minimum on $\mathcal{N}_\lambda^{-}$.
\end{lemma}
\begin{proof}
We know from Corollary \ref{cor:2} that $\inf_{\mathcal{N}_\lambda^{-}} E_\lambda >0$. Let $\{v_k\}_k \subset \mathcal{N}_\lambda^{-}$ be a minimizin sequence, i.e.
\[
\lim_{k \to +\infty} E_\lambda (v_k) = \inf_{\mathcal{N}_\lambda^{-}} E_\lambda >0.
\]
By coercivity on $\mathcal{N}_\lambda^{-}$ we may assume without loss of generality that $\{v_k\}_k$ converges weakly to some $v_\lambda$ in $H^1(\mathbb{R}_{+}^{N+1},y^{1-2\alpha})$. Now,
\[
E_\lambda(v_k) = \left( \frac{1}{2} - \frac{1}{p} \right) \left\| v_k \right\|_{H^1(\mathbb{R}_{+}^{N+1},y^{1-2\alpha})}^2 + \left( \frac{1}{p} - \frac{1}{q} \right) \int_{\mathbb{R}^N} c(x) |\tr v_k(x)|^{q}\, dx.
\]
By the assumptions on $c$, Proposition \ref{prop:1} and Remark \ref{rem:2}
\[
\lim_{k \to +\infty} \int_{\mathbb{R}^N} c(x) |\tr v_k(x)|^{q}\, dx = \int_{\mathbb{R}^N} c(x) |\tr v_\lambda(x)|^{q}\, dx.
\]
Hence $\int_{\mathbb{R}^N} c(x) |\tr v_\lambda(x)|^{q}\, dx>0$. As a consequence by Lemma \ref{lem:4}, the fiber function $\varphi_{v_\lambda}$ possesses a global maximum at some $\tilde{t}=\tilde{t}(v_l)$ such that $\tilde{t} v_\lambda \in \mathcal{N}_\lambda^{-}$. 
On the other hand, $v_k \in \mathcal{N}_\lambda^{-}$ already implies that $1$ is a global maximum point for $\varphi_{v_k}$, namely $\varphi_{v_k}(t) \leq \varphi_{v_k}(1)$ for every $t>0$.

If \(\{v_k\}_k\) does not converge strongly to $v_\lambda$ in $H^1(\mathbb{R}_{+}^{N+1},y^{1-2\alpha})$, then
\begin{multline*}
E_\lambda (\tilde{t}v_\lambda) = \frac{\tilde{t}^2}{2} \int_{\mathbb{R}_{+}^{N+1}} y^{1-2\alpha} \left( |\nabla v_\lambda|^2 + v_\lambda^2 \right)\, dx\, dy - \frac{\lambda \tilde{t}^{p}}{p} \int_{\mathbb{R}^N} b(x) |\tr v_\lambda(x)|^{p}\, dx \\
{}- \frac{\tilde{t}^{q}}{q} \int_{\mathbb{R}^N} c(x) |\tr v_\lambda(x)|^{q}\, dx \\
< \liminf_{k \to +\infty} \frac{\tilde{t}^2}{2} \int_{\mathbb{R}_{+}^{N+1}} y^{1-2\alpha} \left( |\nabla v_k|^2 + v_k^2 \right)\, dx\, dy - \frac{\lambda \tilde{t}^{p}}{p} \int_{\mathbb{R}^N} b(x) |\tr v_k(x)|^{p}\, dx \\
{}- \frac{\tilde{t}^{q}}{q} \int_{\mathbb{R}^N} c(x) |\tr v_k(x)|^{q}\, dx \\
\leq \lim_{k \to +\infty} E_\lambda (\tilde{t}v_k) \leq \lim_{k \to +\infty} E_\lambda (v_k) = \inf_{\mathcal{N}_\lambda^{-}} E_\lambda,
\end{multline*}
a contradiction. We have proved that $v_k \to v_\lambda$ strongly in $H^1(\mathbb{R}_{+}^{N+1},y^{1-2\alpha})$ as $k \to +\infty$, and $\mathcal{N}_\lambda^0 = \emptyset$ finally implies that $v_\lambda \in \mathcal{N}_\lambda^{-}$.
\end{proof}
We are ready to prove our main result.
\begin{theorem}
Let $N>2\alpha$ and $1<p<2<q<2N/(N-2\alpha)$. Suppose that both $|b|$ and $|c|$  satisfy the compactness conditions (K) and that $b$ and $c$ change sign
according to Definition \ref{def:2}.
    Then there exists $\lambda_0>0$ such that for every $0<\lambda<\lambda_0$ equation (\ref{eq:14}) has at least two non-negative solutions.
\end{theorem}
\begin{proof}
We only have to prove that the two minima  of $E_\lambda$ on $\mathcal{N}_\lambda^{+}$ and on $\mathcal{N}_\lambda^{-}$ respectively can be chosen to be non-negative. To this aim, we let
\begin{align*}
f(x,t) &= b(x)|t|^{p-2}t + c(x) |t|^{q-2}t \\
f_{+}(x,t) &= \begin{cases}
                  f(x,t) &\text{if $t \geq 0$} \\
                  0 &\text{if $t<0$}
                  \end{cases} \\
F_{+}(x,t) &= \int_0^t f_{+}(x,s)\, ds.
\end{align*}
To the functional 
\[
E_\lambda^{+} \colon v \mapsto \|v\|^2_{H^1(\mathbb{R}_{+}^{N+1},y^{1-2\alpha})} - \int_{\mathbb{R}^N} F_{+}(x,v(x))\, dx
\]
we can apply all the previous lemma, and thus for every $0<\lambda<\lambda_0$ the functional $E_\lambda^{+}$ possesses two critical points $v_{\lambda,1} \in \mathcal{N}_\lambda^{+}$ and $v_{\lambda,2} \in \mathcal{N}_\lambda^{-}$. Now, denoting by $v_{\lambda,1}^{-}$ the negative part of $v_{\lambda,1}$, we have
\begin{multline*}
0 = DE_\lambda^{+}(v_{\lambda,1}) v_{\lambda,1}^{-} =
\int_{\mathbb{R}_{+}^{N+1}} y^{1-2\alpha} \left( \nabla v_{\lambda,1} \cdot \nabla v_{\lambda,1}^{-} +v_{\lambda,1} v_{\lambda,1}^{-} \right) \, dx\, dy \\
\geq \int_{\mathbb{R}_{+}^{N+1}} y^{1-2\alpha} \left( |\nabla v_{\lambda,1}^{-} |^2 + |v_{\lambda,1}^{-}|^2 \right) \, dx\, dy.
\end{multline*}
Therefore $v_{\lambda,1} = v_{\lambda,1}^{+} \geq 0$. Since the same computation holds with $v_{\lambda,2}$ in place of $v_{\lambda,1}$, the proof is complete.
\end{proof}

\section{Perspectives and open problems}

As we have seen, the Bessel operator $(I-\Delta)^{\alpha}$ shares many features with the standard fractional Laplacian $(-\Delta)^{\alpha}$. However, the most striking difference between these two operators is that the latter has some scaling properties that the former does not have. As should be clear from (\ref{eq:G}), (\ref{eq:19}) and (\ref{eq:21}), the Bessel operator is not compatible with the semigroup $\mathbb{R}_{+}$ acting on functions as $s \star u \colon x \mapsto u(s^{-1}x)$ for $s>0$. In simpler words, the Bessel operator \emph{does not scale}.

In our opinion, a very challenging problem is that of finding solutions to the fractional scalar field equation
\begin{equation} \label{eq:20}
	(I-\Delta)^\alpha u = g(u) \quad\text{in $\mathbb{R}^N$}
\end{equation}
under the so-called \emph{Berestycki--Lions assumptions} on $g$. In \cites{BerLio1,BerLio2} the case $\alpha=1$ was studied under very mild assumptions on the nonlinear term $g$. In particular, no Ambrosetti--Rabinowitz assumption must be imposed, and no monotonicity assumption like in~(f3).

However, the main tool used in \cites{BerLio1,BerLio2} consists in a clever exploitation of the semigroup action $s \star u \colon x \mapsto u(s^{-1}x)$ for $s>0$. Indeed, solutions are constructed (roughly speaking) by solving the minimization problem
\[
\inf \left\{ \int_{\mathbb{R}^N} |\nabla u|^2 + u^2 \mid \int_{\mathbb{R}^N} G(u)=1 \right\},
\]
where $G$ is the antiderivative of $g$. Then a rescaling from $u$ to a suitable $s \star u$ produces a solution of (\ref{eq:20}) by absorbing the Lagrange multiplier. Let us propose a closely related question by considering the special case of equation (\ref{eq:1})
\begin{equation} \label{eq:24}
(I-\Delta)^\alpha u = |u|^{p-2}u \quad\text{in $\mathbb{R}^N$}.
\end{equation}
We continue to assume that \(1<p<2_\alpha^*\). In \cite{FelmerVergara} it is shown that (\ref{eq:24}) has infinitely many solutions. We claim that it possesses a radially symmetric ground state. Indeed, we consider the quotient
\begin{equation*} 
S = \inf \left\{
\int_{\mathbb{R}^N} \left| \left( I-\Delta)^{\alpha/2} u\right) \right|^2 \,dx \mid u \in L^{\alpha,2}(\mathbb{R}^N),\ \int_{\mathbb{R}^N} |u(x)|^p \, dx =1 
\right\}
\end{equation*}
To overcome the lack of compactness in $\mathbb{R}^N$ we can work in the subspace $L_{\mathrm{rad}}^{\alpha,2}(\mathbb{R}^N)$ consisting of radially symmetric elements of $L^{\alpha,2}(\mathbb{R}^N)$. The fact that $S$ is attained is now an immediate consequence of the compact embedding of $L_{\mathrm{rad}}^{\alpha,2}(\mathbb{R}^N)$ into $L^p(\mathbb{R}^N)$, see \cite{Sickel}. To prove that the minimizer of $S$ can be chosen to be radially symmetric, we use the following inequality: it is probably known, but we could not find a precise recerence in the literature.
\begin{proposition}
Assume that $u \in L^{\alpha,2}(\mathbb{R}^N)$, and let $u^*$ the usual symmetric-decreasing rearrangement of $u$. Then
\begin{equation} \label{eq:26}
\int_{\mathbb{R}^N} \left| \left( I-\Delta \right)^{\alpha/2} u^* \right|^2 \, dx \leq \int_{\mathbb{R}^N} \left| \left( I-\Delta \right)^{\alpha/2} u \right|^2 \, dx
\end{equation}
\end{proposition}
\begin{proof}
We adapt the proof given by \cite{LiebYau}*{Appendix A}, see also \cite{Lieb}. By \cite{Fall}*{Eq. (7.5)},
\begin{equation*}
\kappa_\alpha \int_{\mathbb{R}^N} \left| (-\Delta+1)^{\alpha/2}u(x) \right|^2 \, dx = -2\alpha \int_{\mathbb{R}^N} \frac{\vartheta (\sqrt{|\xi|^2 +1}t)-1}{t^{2\alpha}} |\mathcal{F}u(\xi)|^2 \, d\xi,
\end{equation*}
where $\kappa_\alpha$ is an explicit positive constant and $\vartheta$ solves the ordinary differential equation
\[
\vartheta''+\frac{1-2\alpha}{t}\vartheta' -\vartheta =0, \quad
\vartheta(0)=1.
\]
Moreover,
\begin{equation} \label{eq:27}
\int_{\mathbb{R}^N} \frac{\vartheta (\sqrt{|\xi|^2 +1}t)-1}{t^{2\alpha}} |\mathcal{F}u(\xi)|^2 \, d\xi=
t^{-2\alpha} \int_{\mathbb{R}^N} \left( u(x)P(t,\cdot)*u(x)-u(x)^2 \right)\, dx,
\end{equation}
where
\[
P(t,x)= C'_{N,\alpha} t^{2\alpha} \left( |x|^2 + t^2 \right)^{-\frac{N+2\alpha}{4}} K_{\frac{N+2\alpha}{2}} \left( \sqrt{|x|^2+t^2} \right),
\]
$P(t,\cdot)*u(x)=\int_{\mathbb{R}^N}P(t,x-y)u(y)\, dy$ and $C'_{N,\alpha}$ is another explicit positive constant. Since $K_{\frac{N+2\alpha}{2}}$ is positive and non-increasing (because $K_\nu(s)>0$ and $K'_\nu(s) = -\frac{\nu}{s}K_\nu (s) - K_{\nu -1}(s)<0$ for all $s >0$), we can apply an inequality by Riesz \cite{Lieb}*{Eq. (3.9)} and conclude that the right-hand side of (\ref{eq:27}) decreases if we replace $u$ by $u^*$. We achieve (\ref{eq:26}) by letting $t \to 0$.
\end{proof}
The proof of the following result is now straightforward.
\begin{theorem}
Let $1<p<2_\alpha^*$. Then equation (\ref{eq:24}) possesses a radially symmetric ground state.
\end{theorem}
\begin{remark}
By the results of \cite{Li}, if $u=G_{2\alpha}*u^p$ with $u \in L^r(\mathbb{R}^N)$ and $r > \max \left\{ p,N(p-1)/(2\alpha) \right\}$ then $u$ is radially symmetric and decreasing about some point. The previous result is weaker but much easier to prove. 	
\end{remark}

What happens if we replace the power $|u|^{p-2}u$ in (\ref{eq:24}) with a more general nonlinear term $f(u)$ with subcritical growth at infinity? It is not hard to check that we can solve the minimum problem
\begin{equation*} 
\widetilde{S} = \inf \left\{
\int_{\mathbb{R}^N} \left| \left( I-\Delta)^{\alpha/2} u\right) \right|^2 \,dx \mid u \in L^{\alpha,2}(\mathbb{R}^N),\ \int_{\mathbb{R}^N} F(u(x)) \, dx =1 
\right\}
\end{equation*}
in the framework of radially symmetric functions, but getting rid of the associated Lagrange multiplier is a complicated task. This difficulty is again caused by the lack of scaling invariance for the Bessel fractional operator.

For very similar reasons, it seems that the Bessel operator \( (I-\Delta)^\alpha \) does not produce 
strong variational identities like the local Laplacian or the fractional Laplacian. Since identities 
like Pohozaev's are a consequence of the action of the very same semigroup $s \star u$, we can imagine 
that some troubles arise. Indeed, Felmer \emph{et al.} proved in \cite{FelmerVergara}*{Proposition 5.1}
that the following pointwise identity holds for every $\varphi \in \mathcal{S}(\mathbb{R}^N)$:
\begin{equation} \label{eq:23}
(I-\Delta)^\alpha \left( x \cdot \nabla \varphi \right) = x \cdot \nabla \left[ (I-\Delta)^\alpha \varphi \right] + 2\alpha (I-\Delta)^\alpha \varphi - 2\alpha (I-\Delta)^{\alpha-1} \varphi.
\end{equation}
Here the crucial remark is that such a formula for an ``integration by parts'' involves the \emph{inverse} 
Bessel operator $(I-\Delta)^{\alpha-1}$, since $\alpha -1 <0$. In other words, a direct attempt to extend 
the Pohozaev identity leads to terms that are different in nature from those appearing in the problem. Indeed, if $(I-\Delta)^\alpha u = f(u)$ and $u$ decays sufficiently fast at infinity, we can justify the following formal computation, see \cite{FelmerVergara} for a similar result:
\[
\int_{\mathbb{R}^N} \langle x \mid \nabla u \rangle (I-\Delta)^\alpha u \, dx = \int_{\mathbb{R}^N} \langle x \mid \nabla u \rangle f(u(x))\, dx.
\]
Setting
\[
(I) = \int_{\mathbb{R}^N} \langle x \mid \nabla u \rangle (I-\Delta)^\alpha u \, dx, \quad (II) = \int_{\mathbb{R}^N} \langle x \mid \nabla u \rangle f(u(x))\, dx,
\]
we deduce that
\[
(I) = \left( 2\alpha -N \right) \int_{\mathbb{R}^N} u(x)f(u(x))\, dx + N \int_{\mathbb{R}^N} F(u(x))\, dx - 2\alpha \int_{\mathbb{R}^N} u(x) (I-\Delta)^{\alpha-1} u(x) \, dx
\]
and
\[
(II)=-N \int_{\mathbb{R}^N} F(u(x))\, dx.
\]
Therefore the solution $u$ satisfies the variational identity
\begin{equation} \label{eq:22}
2\alpha \int_{\mathbb{R}^N} u(x) (I-\Delta)^{\alpha-1} u(x) \, dx = 2N \int_{\mathbb{R}^N} F(u(x))\, dx - (N-2\alpha) \int_{\mathbb{R}^N} u(x) f(u(x))\, dx.
\end{equation}
The left-hand side of (\ref{eq:22}) is finite if and only if $u \in L^{\alpha-1,2}(\mathbb{R}^N)$, and this is 
certainly true since $0<\alpha<1$ and therefore $L^{\alpha,2}(\mathbb{R}^N) \subset L^{\alpha-1,2}(\mathbb{R}^N)$, 
see \cite{Lions}*{Equation (7.3)}. Anyway, the identity (\ref{eq:22}) is structurally different than the usual 
Pohozaev identity in the local case $\alpha=1$. 

We observe that for the standard fractional Laplacian $(-\Delta)^\alpha$ the analogous of (\ref{eq:23}) reads
\begin{equation*}
(-\Delta)^\alpha \langle x \mid \nabla u \rangle = 2\alpha (-\Delta)^\alpha u + \langle x \mid \nabla (-\Delta)^\alpha u \rangle,
\end{equation*}
see \cite{Ros-Oton}*{page 602}, so that we get a more useful variational identity.

It would be interesting to understand if classical results like those proved in \cites{BerLio1,BerLio2} 
using variational identities can be somehow extended to Bessel fractional operators.

\end{document}